\documentclass[12pt,twoside]{amsart}

\date{}
\allowdisplaybreaks[4] \footskip=15pt
\renewcommand{\uppercasenonmath}[1]{}

\numberwithin{equation}{section}
\newtheorem{thm}{Theorem}[section]
\newtheorem{prop}[thm]{Proposition}
\newtheorem{df}[thm]{Definition}
\newtheorem{lem}[thm]{Lemma}
\newtheorem{exa}[thm]{Example}
\newtheorem{que}[thm]{Question}
\newtheorem{cor}[thm]{Corollary}
\newtheorem{rem}[thm]{Remark}

\begin{document}
\title{ON SMALL DUAL RINGS}
\author{Liang Shen
}
\address{Department of Mathematics, Southeast University, Nanjing, 210096,
P.R.China \\\indent {\rm Email: lshen@seu.edu.cn} }

 \maketitle \baselineskip=20pt
\begin{abstract}
A ring $R$ is called right (small) dual if every (small) right  ideal of $R$ is a right annihilator. Left (small) dual rings can be defined similarly. And a ring $R$ is called (small) dual if $R$ is left and right (small) dual.  It is proved that $R$ is a dual ring if and only if $R$ is a semilocal and small dual ring.  Several known results are generalized and properties of small dual rings are explored. As applications, some characterizations of QF rings are obtained  through small dualities of rings.
\end{abstract}


\bigskip
\section{Introduction}
Throughout this paper rings are associative with
identity. For a subset $X$ of a ring $R$, the left annihilator of
$X$ in $R$ is ${\bf l}(X)=\{r\in R$: $rx=0$ for all $x\in X\}$.
Right annihilators are defined analogously. We write $J$, $Z_l$,
$Z_r$, $S_l$ and $S_r$ for the Jacobson radical, the left singular
ideal, the right singular ideal, the left socle and the right
socle of $R$, respectively.  Let $M$ be an $R$-module, $N\subseteq_{ess}M$ means that
$N$ is an essential submodule of $M$. We denote by $f=c\cdot$
that $f$ is a left multiplication map by the element $c\in R$.
Right multiplication map can be defined similarly. Let $R$ be a ring. We write
 M$_{R}$ for the category of right $R$-modules. And M$_{n}(R)$ denotes the ring of all $n\times n$ matrices over  $R$.\\
\indent Recall that a ring $R$ is called \emph{right dual} if every right ideal $I$ of $R$ is a right annihilator. Left dual rings can be defined similarly. And $R$ is called a \emph{dual ring} if it is left and right dual. Dual rings were firstly investigated by Baer \cite{B43}, Hall \cite{H39}, Kaplansky \cite{K48}, and discussed in detail by Hajarnavis, Norton \cite{HN85}. There are several generalizations of right dual rings. Recall that a ring  $R$ is called \emph{right Kasch} if every maximal right ideal of $R$ is a right annihilator. Left Kasch rings are defined analogously. And $R$ is called a \emph{Kasch} ring if it is both left and right Kasch.   Kasch rings are named in honor of  Friedrich Kasch. These rings were discussed in detail by Morita \cite{M66}, Nicholson, Yousif \cite{NY03} and so on.  A ring $R$ is called \emph{right} {\rm(}\emph{left}{\rm)} \emph{quasi-dual} if every essential right (left) ideal of $R$ is a right (left) annihilator. $R$ is called \emph{quasi-dual} if it is two-sided quasi-dual. These rings were studied by Page and Zhou \cite{PZ00}. \\
\indent In this article, we discuss the condition that every small right ideal of a ring $R$ is a right annihilator. And these rings are called \emph{right small dual} rings. Recall that a right ideal $I$ of $R$ is $small$ if for any proper right ideal $K$ of $R$, $I+K\neq R$. Left small dual rings can be defined similarly. A ring $R$ is called \emph{small dual} if it is both left and right small dual. By definition, a small right ideal is a ``superfluous'' part of $R$. We will give a short proof in Theorem 2.12 to  show that a right small dual ring $R$  is also a right dual ring if it is semilocal. Hence it is proved in Theorem 3.1 that a ring $R$ is a dual ring if and only if it is a semilocal and small dual ring if and only if it is a semiperfect and small dual ring. Some other properties of small dual rings are explored and several known results are improved. As applications, we obtain some characterizations of QF rings given by small dualities. Recall that a ring $R$ is called a \emph{quasi-Frobenius} {\rm(}\emph{QF}{\rm)} \emph{ring} if it is a right artinian and two-sided dual ring.

\section{Results}
\begin{df}\label{def 2.1}
 {\rm A ring $R$ is called right small dual if every small right ideal $I$ of $R$ is a right annihilator, that is, {\bf rl}($I$)=$I$. Left small dual rings can be defined similarly. $R$ is called small dual if $R$ is left and right small dual.}
\end{df}
\indent Recall that a ring $R$ is called  a \emph{semiprimitive ring} if the Jacobson radical $J$ of $R$ is  zero. Since each small one-sided ideal of $R$ is contained in $J$, it is clear that every semiprimitive ring is small dual.
\begin{exa} \label{Exa 2.2}
The ring of integers $\mathbb{Z}$ is a right small dual ring which is not right dual.
\end{exa}
\begin{proof}
Since $\mathbb{Z}$ is a semiprimitive ring,  $\mathbb{Z}$ is right small dual. It is easy to see that $\mathbb{Z}$ is not right dual.
\end{proof}
\begin{exa} \label{Exa 2.3}
\rm (Bj$\ddot{o}$rk Example) Let $F$ be a field and assume that $a$ $\mapsto
\bar{a}$ is an isomorphism $F$ $\rightarrow \bar{F}\subseteq F$,
where the subfield $\bar{F}\neq F$. Let $R$ denote the left vector
space on basis \{1,$t$\}, and make $R$ into an $F$-algebra by
defining $t^{2}$=0 and $ta$=$\bar{a}t$ for all $a\in F$. Then
\begin{enumerate}
\item $R$ is
left small dual  but not right small dual.
\item M$_{2}(R)$ is not left small dual.
\end{enumerate}
\end{exa}
\begin{proof}
(1). By \cite [Example 2.5, Example 7.3]{NY03}, $R$ is a \emph{left CF} ring (every left ideal is a left annihilator of finite elements in $R$) but not left mininjective. Recall that a ring $R$ is called \emph{left mininjective} if every homomorphism from a minimal left ideal of $R$ to $_{R}R$ can be extended to one from $_{R}R$ to $_{R}R$. Since $R$ is left CF, $R$ is left dual. Thus $R$ is left small dual.
But $R$ is not right small dual, if $R$ is right small dual, by following Proposition 2.6 (4), $R$ is left mininjective.  This is a contradiction.\\
\indent (2). By \cite[Example 2.5]{NY03}, $R$ is a left artinian ring which is not QF. Thus M$_{2}(R)$ is a semilocal ring. If M$_{2}(R)$ is left small dual, by following Theorem 2.12, M$_{2}(R)$ is a left dual ring. Thus M$_{2}(R)$ is right $P$-injective by \cite[Lemma 5.1]{NY03}. Recall that a ring $R$ is called \emph{right P-injective} (\emph{right n-injective}) if every homomorphism from a principal ($n$-generated) right ideal of $R$ to $R_{R}$ can be extended to one from $R_{R}$ to $R_{R}$.  According to \cite[Proposition 5.36]{NY03}, $R$ is right 2-injective. Since $R$ is left artinian and right 2-injective, $R$ is a right 2-injective ring with ACC on right annihilators.  Thus, by \cite[Corollary 3]{R75}, $R$ is a QF ring. This is a contradiction.
\end{proof}
\indent Recall that a right $R$-module $M$ is called $torsionless$ if it can be cogenerated by $R_{R}$. That is, $M$ can be embedded into a direct product of $R_{R}$.
\begin{lem}\label{Lem 2.4}
\cite[Lemma 1.40]{NY03}
If T is a right ideal of R, then ${\bf rl}(T)=T$ if and only if R/T is torsionless.
\end{lem}
\begin{prop}\label{Prop 2.5 }
The following are equivalent for a ring R.
\begin{enumerate}
\item R is right small dual.
\item ${\bf r}({\bf l}(T)\cap Rb)=T+{\bf r}(b)$ for all small right ideals $T$ and  all $b\in R$.
\item R/T is torsionless for every small right ideal T of R.
\end{enumerate}
\end{prop}
\begin{proof}
(1)$\Leftrightarrow$(3) is clear by Lemma 2.4. (1)$\Leftrightarrow$(2) can be obtained from \cite[Lemma 2.1]{NY98}. For the sake of
completeness, we provide the proof. (2)$\Rightarrow$(1) is obvious if we set $b=1$. For (1)$\Rightarrow$(2), it is clear that $T$+{\bf r}($b$)$\subseteq$ {\bf r}({\bf l}($T)\cap Rb$). Now we prove the converse. Assume $x\in {\bf r}({\bf l}(T)\cap Rb$) and $y\in {\bf l}(bT)$. It is easy to see $yb\in {\bf l}(T)\cap Rb$. Thus $ybx$=0. This informs that {\bf l}($bT$)$\subseteq${\bf l}($bx$). So $bx\in {\bf rl}(bx)\subseteq {\bf rl}(bT)$. Since $T$ is a small right ideal of $R$, $bT$ is also small. As $R$ is right small dual, {\bf rl}$(bT)=bT$. This implies that $bx\in bT$. So there exists
$t\in T$ such that $bx=bt$. Thus $x-t\in {\bf r}(b)$. Therefore, $x=t+x-t\in T+{\bf r}(b)$.
\end{proof}
\begin{prop}\label{Prop 2.6}
Let R be a right small dual ring. Then
\begin{enumerate}
\item {\bf l}$(J)\subseteq_{ess}$ $_{R}R$.
\item J $\subseteq Z_{l}$.
\item {\bf rl}$(T)=T$ for every finitely generated semisimple right ideal T of R.
\item Every homomorphism from a small principal left ideal to R can be extended to one from $_{R}R$ to $_{R}R$. In particularly,
R is left mininjective.
\end{enumerate}
\end{prop}
\begin{proof}
(1) and (2) are obtained by \cite[Proposition 3.2]{ZLG10}.\\
\indent We will prove (3) by induction on G.dim ($T$) which is the Goldie (or called uniform) dimension of $T$. If G.dim ($T$)=1, then $T$ is a minimal right ideal of $R$. By \cite[Lemma 10.22]{L91}, $T$ is either nilpotent or a direct summand of $R_{R}$. If $T$ is nilpotent then it is contained in $J$, which is the biggest small right (or left) ideal of $R$. This informs that $T$ is  small. Thus $T$ is a right annihilator. If $T$ is a direct summand of $R_{R}$, it is clear that $T$ is a right annihilator. Now we assume that {\bf rl}{\rm($T$)}=$T$ for every semisimple right ideal $T$ of $R$ with G.dim ($T)\leq n$, where $n\geq 1$. We are going to show that {\bf rl}{\rm($K$)}=$K$ for every  semisimple right ideal $K$ of $R$ with G.dim ($K$)=$n+1$. If $K$ is small then $K$ is already a right annihilator. If $K$ is not small, it is easy to see that  $K=T\oplus eR$ where G.dim ($T$)=$n$, $e^{2}=e\in R$, and $eR$ is a minimal right ideal of $R$. Thus {\bf rl}($K$)={\bf rl}($T\oplus eR$)={\bf r}({\bf l}($T)\cap {\bf l}(e))={\bf r}({\bf l}(T)\cap R(1-e))$. Since $T$ and $(1-e)T$ are semisimple right ideals of $R$ with G.dim ($(1-e)T$)$\leq$ G.dim ($T)=n$. By assumption, {\bf rl}($T$)=$T$ and {\bf rl}($(1-e)T$)=$(1-e)T$.  Applying a similar proof in Proposition 2.5 ``(1)$\Rightarrow$(2)", we have ${\bf r}({\bf l}(T)\cap R(1-e))=T+
{\bf r}(1-e)=T+eR=K$.\\
\indent (4). If $Rt$ is a small principal left ideal of $R$, then $tR$ is a small right ideal of $R$. By assumption, {\bf rl}$(tR)=tR$. Let $f$ be an $R$-linear map from $Rt$ to $_{R}R$. It is clear that ${\bf l}(t)\subseteq {\bf l}(f(t))$. Thus $f(t)\subseteq {\bf rl}(f(t))\subseteq {\bf rl}(t)=tR$. Hence $f=\cdot c$ is a
right multiplication map by some element $c\in R$.  Since a minimal left ideal is either small or a direct summand of $_{R}R$, it is clear that a right small dual ring is left mininjective.
\end{proof}
\begin{que}\label{que 2.7}
If $R$ is a right small dual ring, is ${\bf rl}(S_{r})=S_{r}$?\\
By $(3)$ in the above proposition, the answer is ``Yes" if $S_{r}$ is finitely generated as a right $R$-module.
\end{que}

\begin{thm}\label{thm 2.8}
Let R be a right small dual ring and e an idempotent of R such that ReR=R.  Then eRe is also a right small dual ring.
\end{thm}
\begin{proof}
Let $S=eRe$. Since $ReR=R$, it is clear that $R$ and $S$ are equivalent rings with equivalence functors $F=(-\otimes_{R} Re): {\rm M}_{R}\rightarrow {\rm M}_{S}$ and $G=(-\otimes_{S} eR): {\rm M}_{S}\rightarrow {\rm M}_{R}$. According to Proposition 2.5, we will show that $S/K$ is torsionless for every small right ideal $K$ of $S$. It is clear that  $G(S/K)\cong G(S)/G(K)$ and $G(K)$ is small in $G(S)\cong eR$. Thus, by
 \cite[Proposition 21.6 (3)]{AF92}, $S/K$ is cogenerated by $S$ if and only if $eR/T$ is cogenerated by $eR$, where $T$ is a small submodule of $eR$.
Now we only need to prove that $eR/T$ can be cogenerated by $eR$. It is clear that $T$ is also a small right ideal of $R$. As $R$ is right small dual, by Proposition 2.5, $eR/T\hookrightarrow R/T$ which is a torisonless right $R$-module. This informs that $eR/T$ can be cogenerated by $R_{R}$. Since $ReR=R$, $eR$ is a progenerator of M$_{R}$. It is easy to see that $R$ can be cogenerated by $eR$. Hence $eR/T$ can also be  cogenerated by $eR$. Thus  $S$ is a right small dual ring.
\end{proof}
\indent The following example shows that if $R$ is small dual then $eRe$ may not be  small dual, where $e^{2}=e\in R$.
\begin{exa}\label{exa 2.9}
Let $R$ be the algebra of matrices over a field $K$ of the form
\begin{center}
{$R=\left[
\begin{array}{ccccccc}
a & x & 0 & 0& 0 & 0 \\
0 & b & 0 & 0& 0 & 0 \\
0 & 0 & c & y& 0 & 0 \\
0 & 0 & 0 & a& 0 & 0 \\
0 & 0 & 0 & 0& b & z \\
0 & 0 & 0 & 0& 0 & c \\
\end{array}
\right]$, $a,b,c,x,y,z\in K$.}
\end{center}
 Set $e=e_{11}+e_{22}+e_{44}+e_{55}$, which is a sum of canonical matrix
units. Then $e$ is an idempotent of $R$ such that
$ReR\neq R$.  $R$ is a  small dual ring. But $eRe$ is not
small dual.
\end{exa}
\begin{proof}
\cite[Example 9]{K95} informs that $R$ is a QF ring and $eRe$
is not a dual ring. Since $R$ is QF, $R$ is small dual and $eRe$ is a semilocal ring. If $eRe$ is small dual, by following Theorem 2.12, $eRe$ is a dual ring. This contradiction implies that $eRe$ is not small dual.
\end{proof}
\indent For an $R$-module $N$,   we write $N^{m\times n}$ for the set of all formal $m\times n$ matrices whose entries are elements of $N$.  To be convenient, write $N^{n}=N^{1\times n}$ and $N_{n}=N^{n\times 1}$. Let $M_{R}$, $_{R}N$ be $R$-modules.
If $X\subseteq M^{l\times m}$, $S\subseteq R^{m\times n}$ and $Y\subseteq N^{n\times k}$. Define
$${\bf r}_{R^{m\times n}}(X)=\{s\in R^{m\times n}: xs=O_{l\times n}, \forall x\in X\},$$
$${\bf l}_{R^{m\times n}}(Y)=\{s\in R^{m\times n}: sy=O_{m\times k}, \forall y\in Y\}.$$
 Example 2.3 shows that a  matrix ring over a left small dual ring may
not be left small. Next we have

\begin{thm}\label{thm 2.10}
The following are equivalent for a ring R.
\begin{enumerate}
\item {\rm$M$}$_{n}(R)$ is a right small dual ring.
\item $\forall K_{R}\leq J_{n}$, every n-generated right $R$-module $R_{n}/K$ is torsionless.
\item $\forall K_{R}\leq J_{n}$ and $b\in R_{n}$, {\bf l}$_{R^{n}}(K)\subseteq {\bf l}_{R^{n}}(b)$ implies $b\in K$.
\item $\forall K_{R}\leq J_{n}$, {\bf r}$_{R_{n}}{\bf l}_{R^{n}}(K)=K$.
\end{enumerate}
\end{thm}
\begin{proof}
(1)$\Leftrightarrow$(2). Set $S=M_{n}(R)$ and $P=R_{n}$. Then $P$ is a left $S$-right $R$-bimodule. By \cite[Theorem 22.2, Corollary 22.4]{AF92},  $R$ and $S$ are equivalent rings with equivalence functors $F=Hom_{R}(_{S}P_{R},-): {\rm M}_{R}\rightarrow {\rm M}_{S}$ and $G=(-\otimes _{S}P): {\rm M}_{S}\rightarrow {\rm M}_{R}$. By Proposition 2.5, $S$ is right small dual if and only if  $S/K$ is cogenerated by $S_{S}$ for every small right ideal $K$ of $S$. Since $G(S/K)\cong G(S)/G(K)$, by
 \cite[Proposition 21.6 (3)]{AF92}, $S/K$ is cogenerated by $S$ if and only if $G(S)/G(K)$ is cogenerated by $G(S)\cong P$. According to \cite[Proposition 21.6 (5)]{AF92}, $K$ is a small right ideal of $S$ if and only if  $G(K)$ is a small submodule of   $G(S)$. Now replace $G(S)$ by $P$. Since $J_{n}$ is the radical of $R_{n}$, it is clear that (1)$\Leftrightarrow$(2).\\
 \indent (2)$\Rightarrow$(3). Let $K_{R}\leq J_{n}$ and {\bf l}$_{R^{n}}(K)\subseteq {\bf l}_{R^{n}}(b)$ for some $b\in R_{n}$. If $b\notin K$ then $0\neq \overline{b}=b+K\in R_{n}/K$. Since $R_{n}/K$ is torsionless,  there exists an $R$-linear map $f: R_{n}/K\rightarrow R_{R}$ such that $f(\overline{b})\neq 0$. Now set $g: R_{n}\rightarrow R$ with $g(x)=f(\overline{x}), \forall~ x\in R_{n}$. Then $g$ is a right $R$-linear map from $R_{n}$ to $R_{R}$ with $g(K)=0$ and $g(b)\neq 0$. It is not difficult to see that  there exists $c\in R^{n}$ such that $g(x)=cx, \forall ~x\in R_{n}$. Thus $cK=g(K)=0$ implies that $c\in {\bf l}_{R^{n}}(K)\subseteq {\bf l}_{R^{n}}(b)$. So $g(b)=cb=0$. This is a contradiction.\\
  \indent (3)$\Rightarrow$(2). Let $K_{R}\leq J_{n}$. By \cite[Proposition 8.10 (2)]{AF92}, it is equivalent to prove that for each $0\neq\overline{b}\in R_{n}/K$, there exists a right $R$-homomorphism from $R_{n}/K$ to $R_{R}$ such that $g(\overline{b})\neq 0$. If $0\neq\overline{b}$ then $b\not\in K$.  (3) informs that {\bf l}$_{R^{n}}(K) \not\subseteq{\bf l}_{R^{n}}(b)$. Thus there exists $c\in R_{n}$ such that $cK=0$ and $cb\neq 0$. Now define a map $g: R_{n}/K\rightarrow R_{R}$ with $g(\overline{x})=cx$. It is clear that $g$ is a well-defined right $R$-linear map from $R_{n}/K$ to $R_{R}$ with $g(\overline{b})=cb\neq 0$.\\
  \indent (3)$\Rightarrow$(4). Let $K_{R}\leq J_{n}$. It is obvious that $K\subseteq {\bf r}_{R_{n}}{\bf l}_{R^{n}}(K)$. If $b\in {\bf r}_{R_{n}}{\bf l}_{R^{n}}(K)$ then
  {\bf l}$_{R^{n}}(K)\subseteq {\bf l}_{R^{n}}(b)$. Applying (3), $b\in K$.\\
  \indent (4)$\Rightarrow$(3). Let $K_{R}\leq J_{n}$. If {\bf l}$_{R^{n}}(K)\subseteq {\bf l}_{R^{n}}(b)$ then $b\in {\bf r}_{R_{n}}{\bf l}_{R^{n}}(b)\subseteq {\bf r}_{R_{n}}{\bf l}_{R^{n}}(K)=K$.

\end{proof}
\indent If $R$ is a local ring, then every proper right ideal of $R$ is small. It is clear that a local and right small dual ring is right dual. In Theorem 2.12, we will show that a semilocal and right small dual ring is  right dual.
Recall that a right ideal $I$ of a ring $R$ has a $weak$ $supplement$ in $R$ if there exists a right ideal $K$ of $R$
such that $I+K=R$ and $I\cap K$ is a small right ideal of $R$.
\begin{lem}\label{lem 2.11}
The following are equivalent for a ring R.
\begin{enumerate}
\item R is right semilocal.
\item Every right ideal I of R has a weak supplement in R.
\item Every left ideal I of R has a weak supplement in R.
\end{enumerate}
\end{lem}
\begin{proof}
See \cite[Corollary 3.2]{L99}.
\end{proof}

\begin{thm}\label{thm 2.12}
If R is a semilocal ring, then R is right small dual if and only if R is right dual.
\end{thm}
\begin{proof}
We only need to prove the necessity.
Let $I$ be any right ideal of $R$. We will show that $I$ is a right annihilator. According to Lemma 2.4, it is equivalent to prove
 that $R$/$I$ is torsionless. Since $R$ is semilocal, by Lemma 2.11, $I$ has a weak supplement $K$ in $R$. That is, $I$+$K$=$R$ and $I\cap K$ is small.
 As $R$ is right small dual, by Proposition 2.5,  $\frac{R}{I\cap K}\hookrightarrow R_{R} ^{A}$ for some set $A$.
 Thus we have
 \begin{center}
 $\frac{R}{I}=\frac{I+K}{I}\cong\frac{K}{I\cap K}\hookrightarrow\frac{R}{I\cap K}\hookrightarrow R_{R} ^{A}$.
 \end{center}
This shows that  $R$/$I$ is torsionless. Hence $I$ is a right annihilator.
\end{proof}
\indent Since semiperfect rings and left (or right) perfect rings are semilocal, we have
\begin{cor}\label{cor 2.13}
\cite[Lemma 1.5]{Y97}
Suppose R is semiperfect and {\bf rl}(K)=K for every small right ideal K of R. Then {\bf rl}(T)=T for every right ideal T of R.
\end{cor}
\begin{cor}\label{cor 2.14}
\cite[Theorem 3.11]{ZLG10}
If R is a left (or right) perfect, then R is small dual if and only if R is dual.
\end{cor}
 \indent Next we will get a similar result as that in Theorem 2.12 when the small right ideals of $R$ are restricted to be $n$-generated. We define a ring $R$ to be \emph{J-regular} if \emph{R/J} is a von Neumann regular ring.
\begin{lem}\label{lem 2.15}
\cite[Proposition 2.3]{S10}
The following are equivalent for a ring R.\\
{\rm (1)} $R$ is J-regular.\\
{\rm (2)} Every principal  right (or left) ideal of $R$ has a weak supplement in R.\\
{\rm (3)} Every finitely generated right (or left) ideal of $R$ has a weak supplement in R.
\end{lem}
\begin{lem}\label{lem 2.16}
\cite[Lemma 2.9]{S10}
Let R be a ring, $b, r_{i}, a_{i}\in R$, $i=1,2,\ldots,n$, such that $b+\sum_{i=1}^{n}a_{i}r_{i}=1$.
Then $bR\cap \sum_{i=1}^{n}a_{i}R=\sum_{i=1}^{n}ba_{i}R$.
\end{lem}
\indent Using  a similar proof in Theorem 2.12, we have
\begin{thm}\label{thm 2.17}
Let R be a J-regular ring. For a given integer $n\geq 1$,   every n-generated small right ideal of $R$ is a right annihilator
if and only if every n-generated right ideal of R is a right annihilator.
\end{thm}
\begin{proof}
We only need to prove the necessity. Let $I=a_{1}R+\cdots +a_{n}R$ be an $n$-generated right ideal of $R$, by Lemma 2.4, we need to show that $R/I$ is torsionless as a right $R$-module. Since $R$ is right $J$-regular,  by Lemma 2.15, $I$ has a weak supplement in $R$. Thus, there exists a right ideal $K$ of $R$ such that $I+K=R$ and $I\cap K\subseteq J$. It is easy to see there are $r_{1}, \ldots, r_{n}\in R$, $b\in K$ such that $b+\sum_{i=1}^{n}a_{i}r_{i}=1$ and
$I\cap bR\subseteq I\cap K\subseteq J$. Therefore, $I\cap bR$ is a small right ideal of $R$. By Lemma 2.16, $I\cap bR=\sum_{i=1}^{n}ba_{i}R$ is $n$-generated. By the assumption,  $I\cap bR$ is a right annihilator. According to  Lemma 2.4, $R/(I\cap bR)$ is torsionless. Thus
\begin{center}
 $\frac{R}{I}=\frac{I+bR}{I}\cong\frac{bR}{I\cap bR}\hookrightarrow\frac{R}{I\cap bR}\hookrightarrow R_{R} ^{A}$ for some set $A$.
 \end{center}

\end{proof}
\indent Recall that a ring $R$ is called \emph{semiregular} if $R$ is $J$-regular and idempotents can lift modulo $J$.  It is clear that a semiregular or semilocal ring is $J$-regular, so we have
\begin{cor}\label{cor 2.18}
Let R be a semiregular or semilocal ring. For a given integer $n\geq 1$,   every n-generated small right ideal of $R$ is a right annihilator
if and only if every n-generated right ideal of R is a right annihilator.
\end{cor}
\begin{cor}\label{cor 2.19}
\cite[Proposition 3.6]{ZLG10}
Suppose that R is a semiregular and right small dual ring. Then
\begin{enumerate}
\item R is a left C2 ring.
\item $J=Z_{l}$.
\end{enumerate}
\end{cor}
\begin{proof}
By the above corollary, every principle right ideal of $R$ is a right annihilator. Thus $R$ is left $P$-injective by \cite[Lemma 5.1]{NY03}. Then the left is clear by \cite[Proposition 5.10, Theorem 5.14]{NY03}.
\end{proof}
\section{Applications}
\indent\indent Recall that a ring $R$ is called \emph{right simple injective} if every homomorphism $f$ from a right ideal $I$ of $R$ to $R_{R}$ with simple image can be extended to one from $R_{R}$ to $R_{R}$. Left simple injective rings can be defined analogously.
\begin{thm}\label{thm 3.1}The following are equivalent for a ring R.
\begin{enumerate}
\item R is a dual ring.
\item R is a semilocal and small dual ring.
\item R is a semiperfect and small dual ring.
\item R is a two-sided Kasch and two-sided simple injective ring.
\end{enumerate}
\end{thm}
\begin{proof}
(1)$\Leftrightarrow$(2)$\Leftrightarrow$(3). By \cite[Theorem 3.9]{HN85}, A dual ring is a semiperfect ring.  Since a semiperfect ring is semilocal, by Theorem 2.12, $R$ is dual if and only if $R$ is semilocal and small dual if and only if $R$ is semiperfect and small dual.\\
\indent (1)$\Leftrightarrow$(4). See \cite[Theorem 6.18]{NY03}.
\end{proof}
\begin{prop}\label{prop 3.2}
 Let R be a right small dual ring. Then J is nilpotent if R satisfies any one of the following chain conditions:
\begin{enumerate}
\item R satisfies ACC on left annihilators.
\item R satisfies ACC on essential left ideals.
\item R satisfies ACC on essential right ideals.
\end{enumerate}
\end{prop}
\begin{proof}
(1) and (2) are obtained by \cite[Corollary 3.3, Theorem 3.4]{ZLG10}.\\
 \indent (3). If $R$ satisfies ACC on essential right ideals, by \cite[Lemma 2]{DHW89}, $R/S_{r}$ is right noetherian. Given chain $J\supseteq J^{2}\supseteq J^{3}\supseteq \cdots$, we have chain of right ideals $S_{r}\subseteq {\bf l}(J)\subseteq {\bf l}(J^{2})\subseteq {\bf l}(J^{3})\subseteq \cdots$. Since $R/S_{r}$ is right noetherian, there exists integer $n\geq 1$
 such that {\bf l}($J^{n}$)={\bf l}($J^{n+1}$). Thus {\bf rl}($J^{n}$)={\bf rl}($J^{n+1}$). As $R$ is right small dual, $J^{n}={\bf rl}(J^{n})={\bf rl}(J^{n+1})=J^{n+1}$. For each right ideal $I$ of $R$, set $\overline{I}=(I+S_{r})/S_{r}$. Then $\overline{R}$ is noetherian. Thus $\overline{J^{n}}$ is finitely generated and $\overline{J^{n}}=\overline{J^{n+1}}$=$\overline{J^{n}}J$. By Nakayama's Lemma (\cite[Corollary 15.13]{AF92}), $\overline{J^{n}}=\overline{0}$. Hence $J^{n}\subseteq S_{r}$. This informs that $J^{n+1}\subseteq S_{r}J=0$.
\end{proof}
\begin{que}\label{que 3.3}
Let R be a right small dual ring. If R satisfies ACC on right annihilators, is J nilpotent?
\end{que}
\begin{prop}\label{prop 3.4}
Let R be a small dual ring. Then
\begin{enumerate}
\item $R$ is left and right mininjective. In particular, $S_{r}=S_{l}$.
\item ${\bf l}(J)\subseteq_{ess}$ $_{R}R$ and ${\bf r}(J)\subseteq_{ess}$ $R_{R}$.
\item J $\subseteq Z_{l}$ and J $\subseteq Z_{r}$.
\item ${\bf rl}(T)=T$ and ${\bf lr}(K)=K$ for every finitely generated semisimple right ideal T and left ideal K of R.
\end{enumerate}
\end{prop}
\begin{proof}
By Proposition 2.6 (4), $R$ is left and right mininjective. Thus, by \cite[Theorem 2.21]{NY03}, $S_{r}=S_{l}$. (2), (3) and (4) are directly obtained by Proposition 2.6.
\end{proof}
\indent At last we show some characterizations of QF rings given by small dualities of ring.
\begin{lem}\label{lem 3.5}
\cite[Proposition 3]{NY97}
Suppose R is a left perfect, left and right simple injective ring. Then R is a quasi-Frobenius ring.
\end{lem}
\begin{lem}\label{lem 3.6}
\cite[Theorem 2.5]{SC06}
If R is a left and right mininjective ring with ACC on right annihilators in which  $S_{r}\subseteq_{ess}R_{R}$, then R is QF.
\end{lem}
\begin{thm}\label{thm 3.7}
The following are equivalent for a ring R.
\begin{enumerate}
\item R is QF.
\item R is a small dual and left perfect ring.
\item R is a small dual ring with ACC on right annihilators in which $S_{r}\subseteq_{ess}R_{R}$.
\item R is a small dual and semilocal ring with ACC on right annihilators.
\item R is a small dual and semilocal ring with ACC on essential right ideals.
\end{enumerate}
\end{thm}
\begin{proof}
It is obvious that (1)$\Rightarrow$(2), (3), (4), (5).\\
\indent (2)$\Rightarrow$(1). If $R$ is a small dual and left perfect ring, by Theorem 3.1, $R$ is left and right simple injective. According to Lemma 3.5, $R$ is QF.\\
\indent (3)$\Rightarrow$(1). By Proposition 3.4, $R$ is left and right mininjective. Then by Lemma 3.6, $R$ is a QF ring.\\
\indent (4)$\Rightarrow$(1) and (5)$\Rightarrow$(1). If $R$ satisfies (4) or (5), by Proposition 3.2,  $J$ is nilpotent. Thus $R$ is a semiprimary ring which is left and right perfect.
According to (2)$\Rightarrow$(1), $R$ is a QF ring.
\end{proof}
\begin{rem}\label{rem 3.8}
 {\rm From the above theorem, we have\\
\indent (a). ``left perfect" in (2) of the above theorem can not be replaced by ``semiperfect". By \cite[Example 3]{NY97}, there is a commutative semiperfect, Kasch and simple injective ring $R$ which is not self-injective. Thus, by Theorem 3.1, $R$ is a small dual and semiperfect ring. But $R$ is not QF.\\
\indent (b). ``$S_{r}\subseteq_{ess}R_{R}$" and ``semilocal" in (3), (4), (5) of the above theorem can not be removed. Because $\mathbb{Z}$ is a small dual and noetherian ring which is not QF.\\
\indent (c). ``Small dual" in (3) of the above theorem can not be replaced by ``right small dual". By \cite[Example 8.16]{NY03}, there is a right Johns ring which is not right artinian. Recall that a ring $R$ is called \emph{right Johns} if $R$ is right dual and right noetherian. And by
\cite[Lemma 8.7 (4)]{NY03}, if $R$ is a right Johns ring then $S_{r}\subseteq_{ess}R_{R}$.\\
\indent (d). (5) gives partial affirmative answers to \cite[Question 2.20]{S12}.
}
\end{rem}

\end{document}